\documentclass[a4paper,reqno]{amsart}

\usepackage{amssymb}
\usepackage{amstext}
\usepackage{amsmath}
\usepackage{amscd}
\usepackage{amsthm}
\usepackage{amsfonts}
\usepackage{enumerate}
\usepackage{graphicx}
\usepackage{latexsym}
\usepackage{mathrsfs}
\usepackage{mathtools}
\usepackage[all]{xy}
\xyoption{all}

\usepackage{pstricks}
\usepackage{lscape}
\usepackage{comment}

\newtheorem{theorem}{Theorem}[section]

\newtheorem{lemma}[theorem]{Lemma}
\newtheorem{proposition}[theorem]{Proposition}
\newtheorem{definition-proposition}[theorem]{Definition-Proposition}

\theoremstyle{definition}
\newtheorem{definition}[theorem]{Definition}

\newtheorem{example}[theorem]{Example}

\newcommand{\Hom}{\operatorname{Hom}\nolimits}
\newcommand{\End}{\operatorname{End}\nolimits}

\renewcommand{\mod}{\mathsf{mod}\hspace{.01in}}

\newcommand{\RHom}{\mathbf{R}\strut\kern-.2em\operatorname{Hom}\nolimits}

\numberwithin{equation}{section}

\usepackage{paralist}

\begin{document}
\title{On the number of tilting modules over a class of Auslander algebras}
.
\thanks{$*$ is the corresponding author. Both of the authors are supported by NSFC
(Nos. 12171207). X. Zhang is supported by the Project Funded by the Priority
Academic Program Development of Jiangsu Higher Education Institutions and the Starting project of Jiangsu Normal University}

\author{Dan Chen}
\address{D. Chen:  School of Mathematics and Statistics, Jiangsu Normal University, Xuzhou, 221116, P. R. China.}
\email{chendan970208@163.com}
\author{Xiaojin Zhang$^*$}
\address{X. Zhang:  School of Mathematics and Statistics, Jiangsu Normal University, Xuzhou, 221116, P. R. China.}
\email{xjzhang@jsnu.edu.cn, xjzhangmaths@163.com}
\maketitle
\begin{abstract} Let $\Lambda$ be a radical square zero algebra of a Dynkin quiver and let $\Gamma$ be the Auslander algebra of $\Lambda$. Then the number of tilting right $\Gamma$-modules is $2^{m-1}$ if $\Lambda$ is of $A_{m}$ type for $m\geq 1$. Otherwise, the number of tilting right $\Gamma$-modules is $2^{m-3}\times14$ if $\Lambda$ is either of $D_{m}$ type for $m\geq 4$ or of $E_{m}$ type for $m=6,7,8$.
\end{abstract}

\section{\bf{Introduction}}
Tilting theory has been essential in the representation theory of finite dimensional algebras since 1970s(see \cite{BGP,BB,HR}). Tilting modules play an important role in tilting theory. So, it is interesting but difficult to classify the tilting modules over a given algebra. There are many algebraists working on this topics. Br\"{u}stle, Hille, Ringel and R\"{o}hrle \cite {BHRR} classified the tilting modules over the Auslander algebra of $K[x]/(x^{n})$. Iyama and Zhang \cite {IZ1} studied tilting modules over Auslander-Gorenstein algebras. Geuenich \cite {G} studied tilting modules of finite projective dimension for the Auslander algebra of $K[x]/(x^{n})$. Zhang \cite {Z2} showed the number of tilting modules over the Auslander algebra of radical square zero Nakayama algebras. Xie, Gao and Huang \cite {XGH} studied the number of tilting modules over the Auslander algebras of radical cube zero Nakayama algebras.  For more recent development on tilting modules, we refer to \cite {AT,K,PS}.

In 2014, Adachi Iyama and Reiten \cite {AIR} introduced the notion of $\tau$-tilting modules as generalizations of tilting modules in terms of mutations. This help us be able to get the tilting modules in terms of support $\tau$-tilting modules. Therefore, it is important to classify support $\tau$-tilting modules for a given algebra. Adachi \cite {A1} classified support $\tau$-tilting modules over Nakayama algebras; Adachi \cite {A2} and Zhang \cite {Z1} studied $\tau$-rigid modules over algebras with radical square zero; Mizuno \cite {M} classified $\tau$-tilting modules over preprojective algebras of Dynkin type; Iyama and Zhang \cite {IZ2} classified $\tau$-tilting modules over the Auslander algebra of $K[x]/(x^{n})$. For more recent development on $\tau$-tilting modules, we can refer to \cite {AiH,DIJ, KK, W, Zi,Z3}.

In this paper, we study tilting modules over the Auslander algebras of radical square zero algebras of Dynkin quiver in terms of $\tau$-tilting theory. By using a bijection over Auslander-Gorenstein algebras built by Iyama and the second author \cite {IZ1}, we can get the number of tilting modules over the Auslander algebras of radical square zero algebras of Dynkin quivers, which extends the results in \cite {Z2}. More precisely, we prove the following main result.

\begin{theorem}\label{1.1} Let $\Lambda$ be a radical square zero algebra of a Dynkin quiver and let $\Gamma$ be the Auslander algebra of $\Lambda$. Then the number of tilting right $\Gamma$-modules is $2^{m-1}$ if $\Lambda$ is of $A_{m}$ type for $m\geq 1$. Otherwise, the number of tilting right $\Gamma$-modules is $2^{m-3}\times14$ if $\Lambda$ is either of $D_{m}$ type for $m\geq 4$ or of $E_{m}$ type for $m=6,7,8$.
\end{theorem}

We show the organization of this paper as follows: In Section 2, we recall some basic preliminaries on tilting modules, $\tau$-tilting modules and Auslander algebras; In Section 3, we prove the main results and give examples to show the main results.

Throughout this paper, all the algebras are finite dimensional basic algebras over an algebraically closed field $K$ and all modules are finitely generated right modules. For a tilting module, we mean the classical tilting module. We use $\tau$ to denote the Auslander-Reitien translation functor. For an algebra $\Lambda$, we use $\mod\Lambda$ to denote the category of finitely generated right $\Lambda$-modules.

\section{\bf{Preliminaries}}

In this section, we recall definitions and basic facts on tilting modules, $\tau$-tilting modules, and Auslander algebras.

 For a module $M$, we use pd$_{\Lambda}M$ and $|M|$ to denote the projective dimension of $M$ and the number of indecomposable direct summand of $M$, respectively. Now we recall the definition of a tilting module \cite{HR}.

\begin{definition}\label{2.1}Let $\Lambda$ be an algebra and $T\in \mod \Lambda$, $T$ is called a tilting module if the following are satisfied:
\begin{enumerate}[\rm(1)]
\item pd$_{\Lambda}{T} \leq1$.
\item Ext$^{i}_{\Lambda}(T,T)=0$, for $i\geq1$.
\item $|T|=|\Lambda|$.
\end{enumerate}
\end{definition}

Now we recall the definition of $\tau$-tilting modules introduced in \cite {AIR}.

\begin{definition}\label{2.2} Let $\Lambda$ be an algebra and $M\in\mod \Lambda$.
\begin{enumerate}[\rm(1)]
\item We call $M$ {\it $\tau$-rigid} if Hom$_{\Lambda}(M,\tau M)=0$.
\item $M$ is called a {\it $\tau$-tilting} module if $M$ is $\tau$-rigid
and $|M|=|\Lambda|$.
\item We call $M$ in $\mod\Lambda$ {\it support $\tau$-tilting} if there exists an idempotent $e$ of $\Lambda$ such that $M$ is a $\tau$-tilting ($\Lambda/(e)$)-module.
\end{enumerate}
\end{definition}

The following lemma \cite[Proposition 2.4]{AIR} on $\tau$-rigid modules is important.

\begin{lemma}\label{2.3} Let $X$ be in $\mod\Lambda$ with a minimal projective presentation $P_{1}\stackrel{d_{1}}{\longrightarrow} P_{0}\stackrel{d_{0}}{\longrightarrow} X\longrightarrow 0$. Then
$X$ is $\tau$-rigid if and only if the map $ \Hom_{\Lambda}(P_{0},X)\stackrel{d_{1}^*}{\longrightarrow} \Hom_{\Lambda}(P_{1},X)$ is surjective, where ${d_1}^*=\Hom_\Lambda(d_{1},X)$.
\end{lemma}

We also need the following definition of support $\tau$-tilting modules.

\begin{definition}\label{2.4} Let $(M,P)$ be a pair with $M\in$ mod$\Lambda$ and $P$ projective.
\begin{enumerate}[\rm(1)]
\item We call $(M,P)$ a $\tau$-rigid pair if $M$ is $\tau$-rigid and Hom$_{\Lambda}(P,M)$=0.
\item We call $(M,P)$ a support $\tau$-tilting (respectively, almost complete support $\tau$-tilting) pair if $(M,P)$ is $\tau$-rigid and $|M|+|P|=|A|$ (respectively, $|M|+|P|=|A|-1$).
\end{enumerate}
\end{definition}

For an algebra $\Lambda$, we use $s\tau$-tilt $\Lambda$ to denote the set of isomorphism class of support $\tau$-tilting modules over $\Lambda$. Denote by $Q$(s$\tau$-tilt$\Lambda$) the support $\tau$-tilting quiver of $\Lambda$. The following lemma in \cite[Corollary 2.38]{AIR} is useful in this paper.

\begin{lemma}\label{2.5} If $Q$(s$\tau$-tilt$\Lambda$) has a finite connected component $C$, then $Q$(s$\tau$-tilt$\Lambda$)=$C$.
\end{lemma}

In the following we recall the definition of Auslander algebras in \cite{ARS}.

\begin{definition}\label{2.6}
An algebra $\Lambda$ is called an Auslander algebra if gl.dim$\Lambda\leq 2$ and $E_{i}(\Lambda)$ is projective for $i=0,1$, where $E_{i}(\Lambda)$ is the $(i+1)$-th term in a minimal injective resolution of $\Lambda$.
\end{definition}

 It is shown in \cite{ARS} that there is one to one bijection between Auslander algebras and algebras of finite representation type in \cite {ARS}. Let $\Lambda$ be an algebra of finite representation type, and $M$ an additive generator of $\mod\Lambda$. Then we call $\Gamma=\End_{\Lambda}M$ {\it the Auslander algebra} of $\Lambda$.

For an algebra $\Lambda$, we use tilt$\Lambda$ to denote the set of isomorphism classes of tilting modules in $\mod\Lambda$.
The following theorem on tilting modules over Auslander algebras in \cite {Z2} is essential in this paper. For more details on this bijection map we refer to \cite{IZ1, J}.

\begin{theorem}\label{2.7}Let $\Lambda$ be an Auslander algebra and $e$ be an idempotent such that $e\Lambda $ is the additive generator of projective-injective modules. Then there is a bijective between the set tilt$A$ of tilting modules over $\Lambda $ and the set $s\tau$-tilt $\Lambda/(e)$ of support $\tau$-tilting modules over $\Lambda/(e)$.
\end{theorem}

Now we recall the definition of Dynkin algebras as follows.

\begin{definition}\label{2.8}
We call an algebra $\Lambda$ of Dynkin type if the quiver of $\Lambda$ is one of the following quivers:

$A_{m}$ ($ m\geq 1$):
$$1\stackrel{a_{1}}{\longrightarrow} 2\stackrel{a_{2}}{\longrightarrow} \cdots \stackrel{a_{m-2}}{\longrightarrow} m-1 \stackrel{a_{m-1}}{\longrightarrow}m $$

$D_{m}$ ($ m\geq 4$):
$$\xymatrix{&  & 1
\ar[d]^{a_{1}} &  & & &\\
& 2 \ar[r]^{a_{2}} & 3\ar[r]^{a_{3}} & 4 \ar[r]^{a_{4}} &
\ar[r] \cdots\ar[r]^{a_{m-2}} &m-1\ar[r]^{a_{m-1}}&m}$$

$E_{6}$:
$$\xymatrix{&  & &
3\ar[d]^{a_{3}} &  &\\
&1 \ar[r]^{a_{1}} & 2 \ar[r]^{a_{2}} & 4\ar[r]^{a_{4}} & 5 \ar[r]^{a_{5}} &
 6}$$

$E_{7}$:
$$\xymatrix{&  & &
3\ar[d]^{a_{3}} & & &\\
&1 \ar[r]^{a_{1}} & 2 \ar[r]^{a_{2}} & 4\ar[r]^{a_{4}} & 5 \ar[r]^{a_{5}} &
 6\ar[r]^{a_{6}} &7}$$

$E_{8}$:
$$\xymatrix{&  & &
3\ar[d]^{a_{3}} & & & &\\
&1 \ar[r]^{a_{1}} & 2 \ar[r]^{a_{2}} & 4\ar[r]^{a_{4}} & 5 \ar[r]^{a_{5}} &
 6\ar[r]^{a_{6}} &7\ar[r]^{a_{7}}&8}$$
\end{definition}

In the following we recall properties of the block decomposition of an algebra from \cite[P92-93, Theorem 1, Proposition 2]{Al}.

\begin{proposition}\label{2.9} Let $A$ be an algebra and $M\in \mod A$. Then
\begin{enumerate}[\rm(1)]
\item $\Lambda$ has a unique decomposition into a direct sum of indecomposable subalgebras, that is $A=A_1\oplus A_2\oplus\dots\oplus Ar$.
\item $M$ has a unique decomposition as $M=M_1\oplus M_2\oplus\dots \oplus M_r$ with $M_i\in \mod A_i$ and $M_iA_j=0$.
\item $\Hom_{A}(M_i,M_j)=0$ for any $i\not=j$.

\end{enumerate}
\end{proposition}

\section{\bf{Main results}}
In this section, we show the number of tilting modules over Auslander algebras of radical square zero Dynkin algebras.
By a straight calculation, one gets the Auslander algebras of radical square zero Dynkin algebras as follows.

\begin{proposition}\label{3.1}
\begin{enumerate}[\rm(1)]
\item Let $\Gamma$ be the Auslander algebra of a radical square zero algebra of type $A_{m}$. Then $\Gamma$ is given by the quiver
$Q_{1}$:
$$ 1\stackrel{a_{1}}{\longleftarrow} 2\stackrel{a_{2}}{\longleftarrow} \cdots \stackrel{a_{2m-3}}{\longleftarrow}2m-2 \stackrel{a_{2m-2}}{\longleftarrow}2m-1$$

$\mathrm{with}$ $\mathrm{the}$ $\mathrm{relations}$:$$a_{2k-1} a_{2k}=0 (1\leq k \leq m-1 ).$$

\item  Let $\Gamma$ be the Auslander algebra of a radical square zero algebra of type $D_m$. Then $\Gamma$ is given by the quiver
$Q_{2}$:
$$
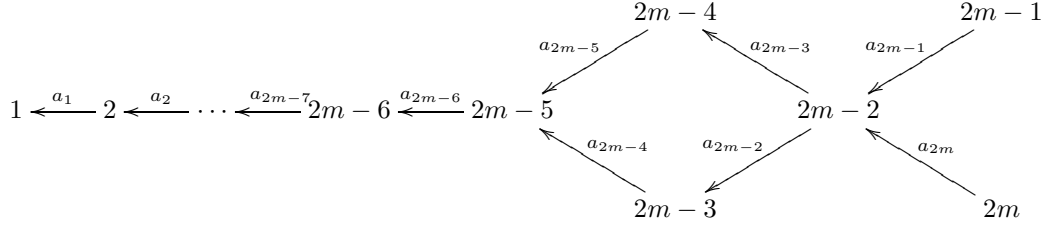
\xymatrix{& & & & & &2m-4\ar[ld]_{a_{2m-5}} & & 2m-1 \ar[ld]_{a_{2m-1}} \\
& 1& 2 \ar[l]_{a_{1}} &\cdots\ar[l]_{a_{2}} & 2m-6 \ar[l]_{a_{2m-7}} &2m-5\ar[l]_{a_{2m-6}} &  & 2m-2 \ar[lu]_{a_{2m-3}}\ar[ld]_{a_{2m-2}} &  \\
& & & & & & 2m-3\ar[lu]_{a_{2m-4}}& & 2m \ar[lu]_{a_{2m}} }$$

$\mathrm{with}$ $\mathrm{the}$ $\mathrm{relations}:$ $a_{1} a_{2}=0, a_{3} a_{4}=0, \cdots, a_{2m-7}a_{2m-6}=0, a_{2m-5}a_{2m-3}= a_{2m-4} a_{2m-2}, $ \\ $ a_{2m-3} a_{2m-1}=0, a_{2m-2} a_{2m}=0 $.

\item  Let $\Gamma_m$ be the Auslander algebra of a radical square zero algebra of type $E_m$ with $m=6,7,8$. Then $\Gamma_{m}$ is given by the quiver $Q_{6}$:
$$
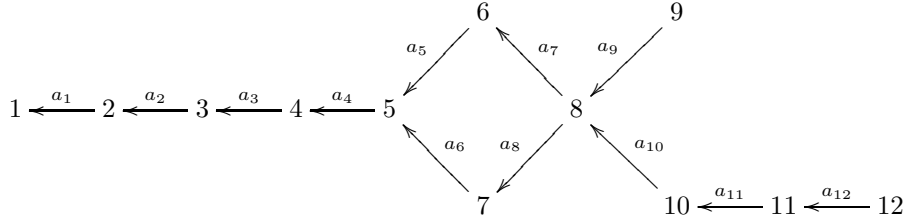
\xymatrix{& & & & & & 6\ar[ld]_{a_{5}} & &9\ar[ld]_{a_{9}} &  & \\
&1 & 2 \ar[l]_{a_{1}} & 3\ar[l]_{a_{2}} & 4 \ar[l]_{a_{3}} &5 \ar[l]_{a_{4}} &  & 8\ar[lu]_{a_{7}}\ar[ld]_{a_{8}}& & & \\
& & & & & &7\ar[lu]_{a_{6}} & &10\ar[lu]_{a_{10}} & 11\ar[l]_{a_{11}} & 12\ar[l]_{a_{12}}}$$

$\mathrm{with}$ $\mathrm{the}$ $\mathrm{relations}:$ $a_{1} a_{2}=0, a_{3} a_{4}=0, a_{5} a_{7}=a_{6} a_{8}, a_{7} a_{9}=0, a_{8} a_{10}=0, a_{11} a_{12}=0 $.

$Q_{7}$:
$$
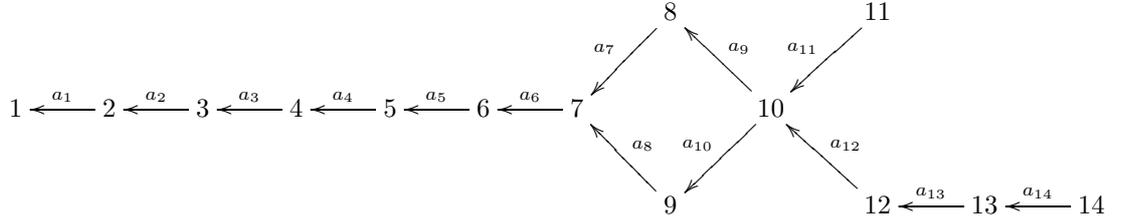
\xymatrix{& & & & & & & & 8\ar[ld]_{a_{7}} & &11\ar[ld]_{a_{11}} &  & \\
&1 & 2 \ar[l]_{a_{1}} & 3\ar[l]_{a_{2}} & 4 \ar[l]_{a_{3}} &5 \ar[l]_{a_{4}} & 6\ar[l]_{a_{5}}& 7 \ar[l]_{a_{6}} &  & 10\ar[lu]_{a_{9}}\ar[ld]_{a_{10}}& & & \\
& & & & & & & &9\ar[lu]_{a_{8}} & &12\ar[lu]_{a_{12}} & 13\ar[l]_{a_{13}} & 14\ar[l]_{a_{14}}}$$

$\mathrm{with}$ $\mathrm{the}$ $\mathrm{relations}:$ $a_{1} a_{2}=0, a_{3} a_{4}=0, a_{5} a_{6}=0, a_{7} a_{9}=a_{8} a_{10}, a_{9} a_{11}=0, a_{10} a_{12}=0, a_{13}a_{14}=0 $.

$Q_{8}$:
$$
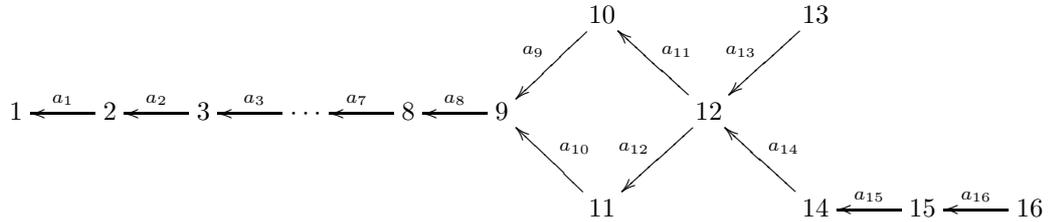
\xymatrix{& & &  & & & & 10\ar[ld]_{a_{9}} & &13\ar[ld]_{a_{13}} &  & \\
&1 & 2 \ar[l]_{a_{1}} & 3\ar[l]_{a_{2}} & \cdots \ar[l]_{a_{3}} & 8 \ar[l]_{a_{7}}& 9 \ar[l]_{a_{8}} &  & 12\ar[lu]_{a_{11}}\ar[ld]_{a_{12}}& & & \\
& & &  & & & &11\ar[lu]_{a_{10}} & &14\ar[lu]_{a_{14}} & 15\ar[l]_{a_{15}} & 16\ar[l]_{a_{16}}}$$

$\mathrm{with}$ $\mathrm{the}$ $\mathrm{relations}:$ $a_{1} a_{2}=0, a_{3} a_{4}=0, a_{5} a_{6}=0, a_{7} a_{8}=0 , a_{9} a_{11}=a_{10} a_{12}, a_{11} a_{13}=0, a_{12} a_{14}=0, a_{15}a_{16}=0 $.

\end{enumerate}
\end{proposition}

To prove the main results we need the following proposition in \cite {Z2}.
\begin{proposition}\label{3.2} Let $\Lambda$ be a semi-simple algebra with $n$ simple modules. Then the number of support $\tau$-tilting $\Lambda$-modules is $2^{n}$.
\end{proposition}
The following proposition on the support $\tau$-tilting modules over direct sums of algebras is essential in this paper.

\begin{proposition}\label{3.3} Let $\Lambda$ be an algebra which can be decomposed as a direct sum of two subalgebras, that is, $\Lambda=\Lambda_{1}\oplus\Lambda_{2}$.
\begin{enumerate}[\rm(1)]

\item For any $M\in\mod\Lambda$, $M$ can be decomposed as $M_1\oplus M_2$ with $M_i\in \mod\Lambda_i$ for $i=1,2$.
\item For any $M\in\mod\Lambda$ with the decomposition $M=M_1\oplus M_2$, $M$ is a support $\tau$-tilting module if both $M_1\in \mod\Lambda_{1}$ and $M_2\in\mod\Lambda_2$ are support $\tau$-tilting modules.
\item For any $M\in\mod\Lambda$ with the decomposition $M=M_1\oplus M_2$, $M$ is a $\tau$-tilting module if both $M_1\in \mod\Lambda_{1}$ and $M_2\in\mod\Lambda_2$ are $\tau$-tilting modules.
\item If $|s\tau$-tilt$\Lambda_{1}|$=$m$ and $|s\tau$-tilt$\Lambda_{2}|$=$n$, then $|s\tau$-tilt$\Lambda|$=$mn$.
\end{enumerate}
\end{proposition}
\begin{proof}
(1) This is a straight result of Proposition \ref{2.9}.

(2) We divide the proof into three steps.

(a) We show $M_1\oplus M_2$ is $\tau$-rigid in $\mod \Lambda$.

Let $P_{1}\stackrel{d_{1}}{\longrightarrow} P_{0}\stackrel{d_{0}}{\longrightarrow} M_1\longrightarrow 0$ be a minimal projective presentation of $M_1\in\mod\Lambda_1$ and let $Q_{1}\stackrel{d_{1}^{'}}{\longrightarrow} Q_{0}\stackrel{d_{0}^{'}}{\longrightarrow} M_2 \longrightarrow 0$  be a minimal projective presentation of $M_2 \in\mod\Lambda_2$. By Proposition \ref{2.9}, one gets a minimal projective presentation of $M=M_1\oplus M_2$ as follows:

$P_{1}\oplus Q_{1} \stackrel{d_{1}\oplus d_{1}^{'}}{\longrightarrow} P_{0}\oplus Q_{0}\stackrel{d_{0}\oplus d_{0}^{'}}{\longrightarrow} M_1\oplus M_2 \longrightarrow 0\ \ \ \ \ \ \ (1)$

Since $M_1$ is support $\tau$-tilting module, then the map $\Hom_{\Lambda_{1}}(P_{0},M_1)\stackrel{(d_{1},M_1)}{\longrightarrow} \Hom_{\Lambda_{1}}(P_{1},M_1)$ is surjective
 by Lemma \ref{2.3}. Similarly, one gets the map $\Hom_{\Lambda_{2}}(Q_{0},M_2)\stackrel{(d_{1}^{'},M_2)}{\longrightarrow}\Hom_{\Lambda_{2}}(Q_{1},M_2)$ is surjective since $M_2$ is support $\tau$-tilting module.  Applying $\Hom_{\Lambda}(-,M_1\oplus M_2)$ to $(1)$, by Proposition \ref{2.9} we have the following exact sequence

 $ 0\longrightarrow \Hom_{\Lambda}(M_1\oplus M_2, M_1\oplus M_2)\stackrel{F_{0}}\longrightarrow \Hom_{\Lambda}(P_{0}\oplus Q_{0},M_1\oplus M_2)\stackrel{F_{1}}{\longrightarrow} \Hom_{\Lambda}(P_{1}\oplus Q_{1},M_1\oplus M_2),$

  \noindent where $F_{i}=(d_{i},M_1)\oplus(d_{i}^{'}, M_2)$ for $i=0,1$. Notice that both the map $(d_{1},M_1)$ and the map $(d_{1}^{'},M_2)$ are surjective, then $(d_{1},M_1)\oplus (d_{1}^{'},M_2)$ is surjective. By Lemma \ref{2.3}, then $M_1\oplus M_2$ is $\tau$-rigid.

(b) Denote by $(M_1,P)$ and $(M_2,Q)$ the support $\tau$-tilting pair in mod$\Lambda_{1}$ and mod$\Lambda_{2}$. We show that $(M_1\oplus M_2,P\oplus Q)$ is $\tau$-rigid pair in mod$\Lambda$.

Since $(M_1,P)$ is a support $\tau$-tilting pair in $\mod \Lambda_{1}$, then $\Hom_{\Lambda_{1}}(P,M_1)=0$ holds. Similarly, one gets $\Hom_{\Lambda_{2}}(Q,M_2)=0$. Then by Proposition \ref{2.9}, $\Hom_{\Lambda}(P\oplus Q,M_1\oplus M_2)\cong \Hom_{\Lambda_{1}}(P,M_1)\oplus \Hom_{\Lambda_{2}}(Q,M_2)\cong 0$, so $(M_1\oplus M_2,P\oplus Q)$ is $\tau$-rigid pair in mod$\Lambda$ by (a).

(c) We show $(M_1\oplus M_2,P\oplus Q)$ is a support $\tau$-tilting pair.

Since $(M_1,P)$ is a support $\tau$-tilting pair, one gets $|M_1|+|P|=|\Lambda_{1}|$. Similarly, the fact $(M_2,Q)$ is a support $\tau$-tilting pair$\in \mod \Lambda_{2}$ implies $|M_2|+|Q|=|\Lambda_{2}|$. So one gets $|M_1|+|M_2|+|P|+|Q|=|\Lambda_{1}|+|\Lambda_{2}|=|\Lambda|$. Then by (b) $(M_1\oplus M_2,P\oplus Q)$ is a support $\tau$-tilting pair in mod $\Lambda$. And hence $M_1\oplus M_2$ is a support $\tau$-tilting module.

(3)This is a straight result of (2).

(4)  Let $|s\tau$-tilt$\Lambda_{1}|$=$m$, $|s\tau$-tilt$\Lambda_{2}|$=$n$, by (2) one gets a finite connected component $\mathcal{C}$ of the quiver of $Q(s\tau$-tilt$\Lambda)$. Then by Lemma \ref{2.5}, we get that $|s\tau$-tilt$\Lambda|$=$mn$.
\end{proof}

The following lemma is also useful.

\begin{lemma}\label{3.4}Let $\Lambda$ be a algebra with $Q$:  $$\xymatrix{ & 4\ar[d] & \\ & 3 &\ar[l]5}$$  Then the number of support $\tau$-tilting $\Lambda$-modules is $14$.
\end{lemma}

\textbf{Proof}. In what follows, we denote a module by its composition factors. Now we can draw the quiver $Q$($s\tau$-tilt$\Lambda$) as follows.
$$
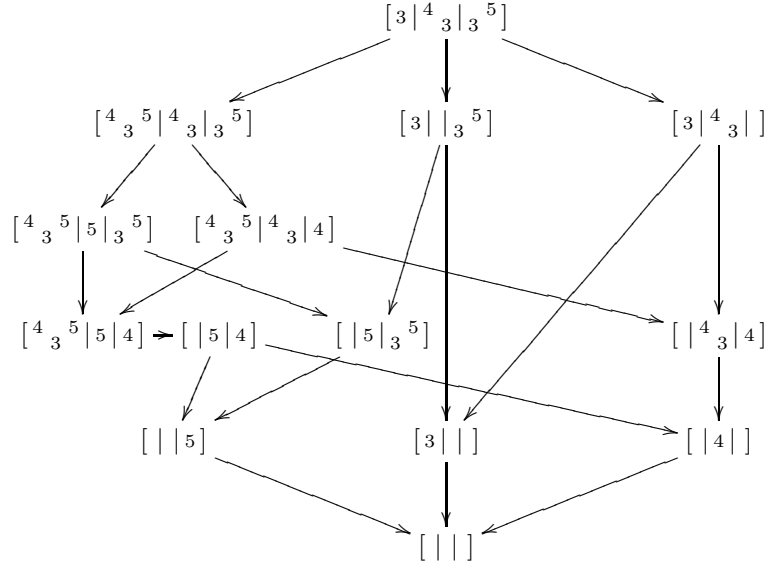
\begin{xy}
0;<3.4pt,0pt>:<0pt,2.5pt>::
(0,0)
*+{ \left[\begin{smallmatrix} 3 \end{smallmatrix}\middle| \begin{smallmatrix} 4\\ &3 \end{smallmatrix}\middle| \begin{smallmatrix} &5\\ 3
\end{smallmatrix}\right]}="1", (-30,-16)
*+{\left[\begin{smallmatrix} 4&&5\\ &3
\end{smallmatrix}\middle| \begin{smallmatrix} 4\\ &3 \end{smallmatrix}\middle| \begin{smallmatrix} &5\\ 3 \end{smallmatrix}\right]}="2", (0,-16)
*+{\left[\begin{smallmatrix} 3 \end{smallmatrix}\middle| \begin{smallmatrix}  \end{smallmatrix}\middle| \begin{smallmatrix} &5\\ 3 \end{smallmatrix}\right]}="3", (30,-16)
*+{\left[\begin{smallmatrix} 3 \end{smallmatrix}\middle| \begin{smallmatrix} 4\\ &3 \end{smallmatrix}\middle| \begin{smallmatrix} \end{smallmatrix}\right]}="4", (-40,-32)
*+{\left[\begin{smallmatrix} 4&&5\\ &3 \end{smallmatrix}\middle|
\begin{smallmatrix} 5 \end{smallmatrix}\middle|
\begin{smallmatrix} &5\\ 3 \end{smallmatrix}\right]}="5", (-20,-32)
*+{\left[\begin{smallmatrix} 4&&5\\&3 \end{smallmatrix}\middle|
\begin{smallmatrix} 4\\ &3  \end{smallmatrix}\middle|
\begin{smallmatrix} 4 \end{smallmatrix}\right]}="6", (-40,-48)
*+{\left[\begin{smallmatrix} 4&&5\\ &3 \end{smallmatrix}\middle| \begin{smallmatrix} 5
 \end{smallmatrix}\middle| \begin{smallmatrix} 4 \end{smallmatrix}\right]}="7", (-25,-48)
*+{\left[\begin{smallmatrix} \end{smallmatrix}\middle|
\begin{smallmatrix} 5 \end{smallmatrix}\middle|\begin{smallmatrix} 4 \end{smallmatrix} \right]}="8",
(-7,-48)
*+{\left[\begin{smallmatrix}  \end{smallmatrix}\middle| \begin{smallmatrix} 5 \end{smallmatrix}\middle| \begin{smallmatrix}
&5\\3 \end{smallmatrix}\right]}="9", (30,-48)
*+{\left[\begin{smallmatrix} \end{smallmatrix}\middle| \begin{smallmatrix} 4\\&3 \end{smallmatrix}\middle| \begin{smallmatrix} 4 \end{smallmatrix}\right]}="10", (-30,-64)
*+{\left[\begin{smallmatrix}
\end{smallmatrix}\middle| \begin{smallmatrix}  \end{smallmatrix}\middle| \begin{smallmatrix} 5\end{smallmatrix}\right]}="11", (0,-64)
*+{\left[\begin{smallmatrix}
3\end{smallmatrix}\middle| \begin{smallmatrix}  \end{smallmatrix}\middle| \begin{smallmatrix}\end{smallmatrix}\right]}="12", (30,-64)
*+{\left[\begin{smallmatrix} \end{smallmatrix}\middle| \begin{smallmatrix}
4\end{smallmatrix}\middle| \begin{smallmatrix}  \end{smallmatrix}\right]}="13", (0,-80)
*+{\left[\begin{smallmatrix}  \end{smallmatrix}\middle| \begin{smallmatrix} \end{smallmatrix}\middle| \begin{smallmatrix}
\end{smallmatrix}\right]}="14",
\ar"1";"2",
\ar"1";"3", \ar"1";"4", \ar"2";"5", \ar"2";"6", \ar"5";"7",
\ar"5";"9", \ar"6";"7", \ar"6";"10", \ar"3";"9", \ar"3";"12",
\ar"4";"10", \ar"4";"12", \ar"7";"8", \ar"9";"11", \ar"10";"13",
\ar"8";"11", \ar"8";"13", \ar"11";"14",\ar"12";"14", \ar"13";"14",
 \end{xy}$$
By Lemma \ref{2.5}, then we can get the number of support $\tau$-tilting $\Lambda$-modules is $14$.

The following result on the number of tilting modules over the Auslander algebras of radical square zero algebras of type $A_n$ has been shown in \cite {Z2}.

\begin{theorem}\label{3.5} Let $\Gamma$ be the Auslander algebra of a radical square zero of type $A_{m}$ with $m\geq 1$. Then the number of tilting $\Gamma$-modules is $2^{m-1}$.
\end{theorem}

For a vertex $i$ in a quiver $Q$, we denote by $P(i)$,$I(i)$ and $S(i)$ the indecomposable projective, injective and simple module according to the $i$. Now we are in a position to show the following part of our main results.

\begin{theorem}\label{3.6} Let $\Lambda$ be a radical square zero algebra of type $D_m$ with $m\geq 4$ and let $\Gamma$ be the Auslander algebra of $\Lambda$. Then the number of tilting $\Gamma$-module is $2^{m-3}\times 14$.
\end{theorem}

\textbf{Proof}. By Proposition \ref{3.1}, we can get the quiver of the algebra $\Gamma$. We also get the indecomposable projective-injective modules are as follows:$P(2)=I(1), P(4)=I(2), P(6)=I(4),\cdots,P(2m-2)=I(2m-6), P(2m-1)=I(2m-3), P(2m)=I(2m-4)$. Take the idempotent $e=e_{2}+e_{4}+e_{6}+\cdots +e_{2m-6} +e_{2m-2}+e_{2m-1}+e_{2m}$. Then $\Gamma/(e)$ is a direct sum of a semi-simple algebra with $m-3$ vertices and an algebra in Lemma \ref{3.4}. Then by Theorem \ref{2.7} , Proposition \ref{3.2}, Proposition \ref{3.3} and Lemma \ref{3.4}, one can get the number of the tilting modules in $\mod\Gamma$ is $2^{m-3}\times 14=112$.

Now we show the number of tilting modules over the Auslander algebras of radical square zero algebras of type $E_m$ for $m=6,7,8$.

\begin{theorem}\label{3.7} Let $\Lambda$ be a radical square zero algebra of type $E_m$ with m=6,7,8. Let $\Gamma$ be the Auslander algebra of $\Lambda$. Then the number of tilting $\Gamma$-module is $2^{m-3}\times 14$.
\end{theorem}

\textbf{Proof}. If $\Lambda$ is of type $E_6$, then  by Proposition \ref{3.1}, we get the quiver of $\Gamma$. Moreover, one can get the indecomposable projective-injective modules are as follows:$P(2)=I(1), P(4)=I(2), P(8)=I(4), P(9)=I(7), P(11)=I(6), P(12)=I(11)$. Take the idempotent $e=e_{2}+e_{4}+e_{8} +e_{9} +e_{11}+e_{12}$. Then the quotient algebra $\Gamma/(e)$ is a direct sum of a semi-simple algebra with $3$ vertices and an algebra in Lemma \ref{3.4}. Then by Theorem \ref{2.7} , Proposition \ref{3.2}, Proposition \ref{3.3} and Lemma \ref{3.4}, we get the number of the tilting modules $\mod\Gamma$ is $2^{3}\times14=112$.

If $\Lambda$ is of type $E_7$, then by Proposition \ref{3.1}, we get the quiver of $\Gamma$. Moreover, one can get the indecomposable projective-injective modules are as follows: $P(2)=I(1), P(4)=I(2), P(6)=I(4), P(10)=I(6), P(11)=I(9), P(13)=I(8), P(14)=I(13)$. Take the idempotent $e=e_{2}+e_{4}+e_{6} +e_{10} +e_{11}+e_{13}+e_{14}$. And hence the quotient algebra $\Gamma/(e)$ is a direct sum of a semi-simple algebra with $4$ vertices and an algebra in Lemma \ref{3.4}. Then by Theorem \ref{2.7} , Proposition \ref{3.2}, Proposition \ref{3.3} and Lemma \ref{3.4}, we get the number of tilting modules in $\mod\Gamma$  is $2^{4}\times14=224$.

If $\Lambda$ is of type $E_8$, then by Proposition \ref{3.1}, we get the quiver of $\Gamma$. Moreover, one gets the indecomposable projective-injective modules are as follows:$P(2)=I(1), P(4)=I(2), P(6)=I(4), P(8)=I(6), P(12)=I(8), P(13)=I(11), P(15)=I(10), P(16)=I(15)$. Take the idempotent $e=e_{2}+e_{4}+e_{6} +e_{8}+e_{12} +e_{13}+e_{15}+e_{16}$. So the quotient algebra $\Gamma/(e)$ is a direct sum of a semi-simple algebra with $5$ vertices and an algebra in Lemma \ref{3.4}. Then by Theorem \ref{2.7} , Proposition \ref{3.2}, Proposition \ref{3.3} and Lemma \ref{3.4}, we can get the number of the tilting modules in the algebra $\Gamma$  is  $2^{5}\times14=448$.

At the end of this paper, we give some examples to show our main results.

\begin{example}\label{3.8}
Let $\Lambda$ be a radical square zero algebra of type $A_3$. Then Auslander algebra $\Gamma$ of $\Lambda$ is given by the quiver $Q: 1\stackrel{u_{1}}{\longleftarrow} 2\stackrel{u_{2}}{\longleftarrow} 3\stackrel{u_{3}}{\longleftarrow} 4\stackrel{u_{4}}{\longleftarrow} 5$. with the relations: $u_{1} u_{2}=0$,  $u_{3} u_{4}=0$.
  Then the tilting $\Gamma$-modules are follows:
    $$T_{1}=\Gamma ,   T_{2}=P(5)\oplus P(4)\oplus S(4)\oplus P(2)\oplus P(1)$$
    $$T_{3}=P(5)\oplus P(4)\oplus P(3)\oplus P(2)\oplus S(2)$$
    $$T_{4}=P(5)\oplus P(4)\oplus S(4)\oplus P(2)\oplus S(2)$$
    $$ 2^{m-1}=2^{3-1}=4$$
\end{example}

\begin{example}\label{3.9}
Let $\Lambda$ be a radical square zero algebra of type $D_4$. Then Auslander algebra $\Gamma$ is given by the quiver $Q$:
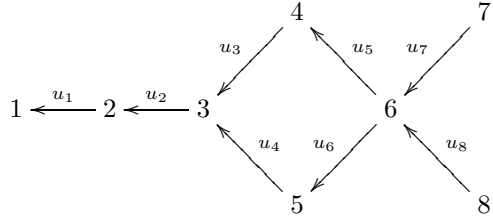
$$ \xymatrix{& & & &4\ar[ld]_{u_{3}} & & 7\ar[ld]_{u_{7}} \\ & 1& 2 \ar[l]_{u_{1}} & 3\ar[l]_{u_{2}} & & 6 \ar[lu]_{u_{5}}\ar[ld]_{u_{6}} &  \\ & & & & 5\ar[lu]_{u_{4}}& & 8\ar[lu]_{u_{8}} }$$
with the relations: $u_{1} u_{2}=0,  u_{3}u_{5}= u_{4} u_{6},  u_{5} u_{7}=0, u_{6} u_{8}=0 $.

The number of support $\tau$ tilting $\Gamma/(e)$-modules are as follows:
$$\mathrm{T_{1}}=1\oplus5, \mathrm{T_{2}}=0\oplus5$$
$$\mathrm{T_{3}}=1\oplus3, \mathrm{T_{4}}=0\oplus3$$
$$\mathrm{T_{5}}=1\oplus4, \mathrm{T_{6}}= 0\oplus4$$
$$\mathrm{T_{7}}=1\oplus0, \mathrm{T_{8}}=0\oplus0$$
$$\mathrm{T_{9}}=1\oplus5\oplus4,\mathrm{T_{10}}= 0\oplus5\oplus4$$
$$\mathrm{T_{11}}=1\oplus5\oplus{ \left[\begin{smallmatrix} 5\\ &3 \end{smallmatrix}\right]}, \mathrm{T_{12}}=0\oplus5\oplus{ \left[\begin{smallmatrix} 5\\ &3 \end{smallmatrix}\right]}$$
$$\mathrm{T_{13}}=1\oplus4\oplus{ \left[\begin{smallmatrix} 4\\ &3 \end{smallmatrix}\right]},\mathrm{T_{14}}= 0\oplus4\oplus{ \left[\begin{smallmatrix} 4\\ &3 \end{smallmatrix}\right]}$$
$$\mathrm{T_{15}}=1\oplus3\oplus{ \left[\begin{smallmatrix} &5\\ 3 \end{smallmatrix}\right]}, \mathrm{T_{16}}=0\oplus3\oplus{ \left[\begin{smallmatrix} &5\\ 3 \end{smallmatrix}\right]}$$
$$\mathrm{T_{17}}=1\oplus3\oplus{ \left[\begin{smallmatrix} 4\\ &3 \end{smallmatrix}\right]},\mathrm{T_{18}}= 0\oplus3\oplus{ \left[\begin{smallmatrix} 4\\ &3 \end{smallmatrix}\right]}$$
$$\mathrm{T_{19}}=1\oplus{ \left[\begin{smallmatrix} 4&&5\\ &3 \end{smallmatrix}\right]}\oplus5\oplus4, \mathrm{T_{20}}=0\oplus{ \left[\begin{smallmatrix} 4&&5\\ &3 \end{smallmatrix}\right]}\oplus5\oplus4$$
$$\mathrm{T_{21}}=1\oplus3\oplus{ \left[\begin{smallmatrix} 4\\ &3 \end{smallmatrix}\right]}\oplus{ \left[\begin{smallmatrix} &5\\ 3 \end{smallmatrix}\right]}, \mathrm{T_{22}}=0\oplus3\oplus{ \left[\begin{smallmatrix} 4\\ &3 \end{smallmatrix}\right]}\oplus{ \left[\begin{smallmatrix} &5\\ 3 \end{smallmatrix}\right]}$$
$$\mathrm{T_{23}}=1\oplus{ \left[\begin{smallmatrix} 4&&5\\ &3 \end{smallmatrix}\right]}\oplus5 \oplus{ \left[\begin{smallmatrix} &5\\ 3 \end{smallmatrix}\right]}, \mathrm{T_{24}}=0\oplus{ \left[\begin{smallmatrix} 4&&5\\ &3 \end{smallmatrix}\right]}\oplus5 \oplus{ \left[\begin{smallmatrix} &5\\ 3 \end{smallmatrix}\right]}$$
$$\mathrm{T_{25}}=1\oplus{ \left[\begin{smallmatrix} 4&&5\\ &3 \end{smallmatrix}\right]}\oplus{ \left[\begin{smallmatrix} 4\\ &3 \end{smallmatrix}\right]}\oplus4,\mathrm{T_{26}}= 0\oplus{ \left[\begin{smallmatrix} 4&&5\\ &3 \end{smallmatrix}\right]}\oplus{ \left[\begin{smallmatrix} 4\\ &3 \end{smallmatrix}\right]}\oplus4$$
$$\mathrm{T_{27}}=1\oplus{ \left[\begin{smallmatrix} 4&&5\\ &3 \end{smallmatrix}\right]}\oplus{ \left[\begin{smallmatrix} 4\\ &3 \end{smallmatrix}\right]}\oplus{ \left[\begin{smallmatrix} &5\\ 3 \end{smallmatrix}\right]}, \mathrm{T_{28}}=0\oplus{ \left[\begin{smallmatrix} 4&&5\\ &3 \end{smallmatrix}\right]}\oplus{ \left[\begin{smallmatrix} 4\\ &3 \end{smallmatrix}\right]}\oplus{ \left[\begin{smallmatrix} &5\\ 3 \end{smallmatrix}\right]}$$
$$ 2^{m-3}\times14=2^{4-3}\times14=28$$
By Theorem \ref{2.7}, then the number of tilting $\Gamma$-modules is 28 ($=2^{4-3}\times 14$).
\end{example}

\begin{example}\label{3.10}
Let $\Lambda$ be a radical square zero algebra of type $E_{6}$. Then Auslander algebra $\Gamma$ is given by the quiver $Q$: 
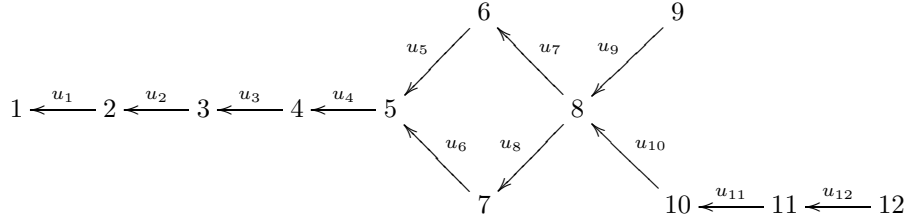
$$\xymatrix{& & & & & & 6\ar[ld]_{u_{5}} & &9\ar[ld]_{u_{9}} &  & \\
&1 & 2 \ar[l]_{u_{1}} & 3\ar[l]_{u_{2}} & 4 \ar[l]_{u_{3}} &5 \ar[l]_{u_{4}} &  & 8\ar[lu]_{u_{7}}\ar[ld]_{u_{8}}& & & \\
& & & & & &7\ar[lu]_{u_{6}} & &10\ar[lu]_{u_{10}} & 11\ar[l]_{u_{11}} & 12\ar[l]_{u_{12}}}$$
with the relations: $u_{1} u_{2}=0, u_{3} u_{4}=0, u_{5} u_{7}=u_{6} u_{8}, u_{7} u_{9}=0, u_{8} u_{10}=0, u_{11} u_{12}=0 $.

The number of support $\tau$ tilting $\Gamma/(e)$-modules are as follows:
$$\mathrm{T_{1}}=1\oplus3\oplus7\oplus0, \mathrm{T_{2}}=0\oplus3\oplus7\oplus0$$
$$\mathrm{T_{3}}=1\oplus3\oplus5\oplus10, \mathrm{T_{4}}=0\oplus3\oplus5\oplus10$$
$$\mathrm{T_{5}}=1\oplus3\oplus6\oplus10, \mathrm{T_{6}}= 0\oplus3\oplus6\oplus10$$
$$\mathrm{T_{7}}=1\oplus3\oplus0\oplus10, \mathrm{T_{8}}=0\oplus3\oplus0\oplus10$$
$$\mathrm{T_{9}}=1\oplus3\oplus7\oplus10, \mathrm{T_{10}}=0\oplus3\oplus7\oplus10$$
$$\mathrm{T_{11}}=1\oplus3\oplus5\oplus0, \mathrm{T_{12}}=0\oplus3\oplus5\oplus0$$
$$\mathrm{T_{13}}=1\oplus3\oplus6\oplus0, \mathrm{T_{14}}= 0\oplus3\oplus6\oplus0$$
$$\mathrm{T_{15}}=1\oplus3\oplus0\oplus0, \mathrm{T_{16}}=0\oplus3\oplus0\oplus0$$
$$\mathrm{T_{17}}=1\oplus0\oplus7\oplus0, \mathrm{T_{18}}=0\oplus0\oplus7\oplus0$$
$$\mathrm{T_{19}}=1\oplus0\oplus5\oplus0, \mathrm{T_{20}}=0\oplus0\oplus5\oplus0$$
$$\mathrm{T_{21}}=1\oplus0\oplus6\oplus0, \mathrm{T_{22}}= 0\oplus0\oplus6\oplus0$$
$$\mathrm{T_{23}}=1\oplus0\oplus0\oplus0, \mathrm{T_{24}}=0\oplus0\oplus0\oplus0$$
$$\mathrm{T_{25}}=1\oplus0\oplus7\oplus10, \mathrm{T_{26}}=0\oplus0\oplus7\oplus10$$
$$\mathrm{T_{27}}=1\oplus0\oplus5\oplus10, \mathrm{T_{28}}=0\oplus0\oplus5\oplus10$$
$$\mathrm{T_{29}}=1\oplus0\oplus6\oplus10, \mathrm{T_{30}}= 0\oplus0\oplus6\oplus10$$
$$\mathrm{T_{31}}=1\oplus0\oplus0\oplus10, \mathrm{T_{32}}=0\oplus0\oplus0\oplus10$$
$$\mathrm{T_{33}}=1\oplus3\oplus7\oplus6\oplus0, \mathrm{T_{34}}= 0\oplus3\oplus7\oplus6\oplus0$$
$$\mathrm{T_{35}}=1\oplus0\oplus7\oplus6\oplus0, \mathrm{T_{36}}= 0\oplus0\oplus7\oplus6\oplus0$$
$$\mathrm{T_{37}}=1\oplus3\oplus7\oplus6\oplus10, \mathrm{T_{38}}= 0\oplus3\oplus7\oplus6\oplus10$$
$$\mathrm{T_{39}}=1\oplus0\oplus7\oplus6\oplus10, \mathrm{T_{40}}= 0\oplus0\oplus7\oplus6\oplus10$$
$$\mathrm{T_{41}}=1\oplus3\oplus5\oplus{ \left[\begin{smallmatrix} &7\\ 5 \end{smallmatrix}\right]}\oplus0, \mathrm{T_{42}}=0\oplus3\oplus5\oplus{ \left[\begin{smallmatrix} &7\\ 5 \end{smallmatrix}\right]}\oplus0$$
$$\mathrm{T_{43}}=1\oplus3\oplus5\oplus{ \left[\begin{smallmatrix} 6\\ &5 \end{smallmatrix}\right]}\oplus0,
\mathrm{T_{44}}= 0\oplus3\oplus5\oplus{ \left[\begin{smallmatrix} 6\\ &5 \end{smallmatrix}\right]}\oplus0$$
$$\mathrm{T_{45}}=1\oplus3\oplus7\oplus{ \left[\begin{smallmatrix} 7\\ &5 \end{smallmatrix}\right]}\oplus0, \mathrm{T_{46}}=0\oplus3\oplus7\oplus{ \left[\begin{smallmatrix} 7\\ &5 \end{smallmatrix}\right]}\oplus0$$
$$\mathrm{T_{47}}=1\oplus3\oplus6\oplus{ \left[\begin{smallmatrix} 6\\ &5 \end{smallmatrix}\right]}\oplus0,
\mathrm{T_{48}}= 0\oplus3\oplus6\oplus{ \left[\begin{smallmatrix} 6\\ &5 \end{smallmatrix}\right]}\oplus0$$
$$\mathrm{T_{49}}=1\oplus0\oplus5\oplus{ \left[\begin{smallmatrix} &7\\ 5 \end{smallmatrix}\right]}\oplus0, \mathrm{T_{50}}=0\oplus0\oplus5\oplus{ \left[\begin{smallmatrix} &7\\ 5 \end{smallmatrix}\right]}\oplus0$$
$$\mathrm{T_{51}}=1\oplus0\oplus5\oplus{ \left[\begin{smallmatrix} 6\\ &5 \end{smallmatrix}\right]}\oplus0,
\mathrm{T_{52}}= 0\oplus0\oplus5\oplus{ \left[\begin{smallmatrix} 6\\ &5 \end{smallmatrix}\right]}\oplus0$$
$$\mathrm{T_{53}}=1\oplus0\oplus7\oplus{ \left[\begin{smallmatrix} 7\\ &5 \end{smallmatrix}\right]}\oplus0, \mathrm{T_{54}}=0\oplus0\oplus7\oplus{ \left[\begin{smallmatrix} 7\\ &5 \end{smallmatrix}\right]}\oplus0$$
$$\mathrm{T_{55}}=1\oplus0\oplus6\oplus{ \left[\begin{smallmatrix} 6\\ &5 \end{smallmatrix}\right]}\oplus0,
\mathrm{T_{56}}= 0\oplus0\oplus6\oplus{ \left[\begin{smallmatrix} 6\\ &5 \end{smallmatrix}\right]}\oplus0$$
$$\mathrm{T_{57}}=1\oplus3\oplus5\oplus{ \left[\begin{smallmatrix} &7\\ 5 \end{smallmatrix}\right]}\oplus10, \mathrm{T_{58}}=0\oplus3\oplus5\oplus{ \left[\begin{smallmatrix} &7\\ 5 \end{smallmatrix}\right]}\oplus10$$
$$\mathrm{T_{59}}=1\oplus3\oplus5\oplus{ \left[\begin{smallmatrix} 6\\ &5 \end{smallmatrix}\right]}\oplus10,
\mathrm{T_{60}}= 0\oplus3\oplus5\oplus{ \left[\begin{smallmatrix} 6\\ &5 \end{smallmatrix}\right]}\oplus10$$
$$\mathrm{T_{61}}=1\oplus3\oplus7\oplus{ \left[\begin{smallmatrix} 7\\ &5 \end{smallmatrix}\right]}\oplus10, \mathrm{T_{62}}=0\oplus3\oplus7\oplus{ \left[\begin{smallmatrix} 7\\ &5 \end{smallmatrix}\right]}\oplus10$$
$$\mathrm{T_{63}}=1\oplus3\oplus6\oplus{ \left[\begin{smallmatrix} 6\\ &5 \end{smallmatrix}\right]}\oplus10,
\mathrm{T_{64}}= 0\oplus3\oplus6\oplus{ \left[\begin{smallmatrix} 6\\ &5 \end{smallmatrix}\right]}\oplus10$$
$$\mathrm{T_{65}}=1\oplus0\oplus5\oplus{ \left[\begin{smallmatrix} &7\\ 5 \end{smallmatrix}\right]}\oplus10, \mathrm{T_{66}}=0\oplus0\oplus5\oplus{ \left[\begin{smallmatrix} &7\\ 5 \end{smallmatrix}\right]}\oplus10$$
$$\mathrm{T_{67}}=1\oplus0\oplus5\oplus{ \left[\begin{smallmatrix} 6\\ &5 \end{smallmatrix}\right]}\oplus10,
\mathrm{T_{68}}= 0\oplus0\oplus5\oplus{ \left[\begin{smallmatrix} 6\\ &5 \end{smallmatrix}\right]}\oplus10$$
$$\mathrm{T_{69}}=1\oplus0\oplus7\oplus{ \left[\begin{smallmatrix} 7\\ &5 \end{smallmatrix}\right]}\oplus10, \mathrm{T_{70}}=0\oplus0\oplus7\oplus{ \left[\begin{smallmatrix} 7\\ &5 \end{smallmatrix}\right]}\oplus10$$
$$\mathrm{T_{71}}=1\oplus0\oplus6\oplus{ \left[\begin{smallmatrix} 6\\ &5 \end{smallmatrix}\right]}\oplus10,
\mathrm{T_{72}}= 0\oplus0\oplus6\oplus{ \left[\begin{smallmatrix} 6\\ &5 \end{smallmatrix}\right]}\oplus10$$
$$\mathrm{T_{73}}=1\oplus3\oplus{ \left[\begin{smallmatrix} 6&&7\\ &5 \end{smallmatrix}\right]}\oplus7\oplus6\oplus0, \mathrm{T_{74}}=0\oplus3\oplus{ \left[\begin{smallmatrix} 6&&7\\ &5 \end{smallmatrix}\right]}\oplus7\oplus6\oplus0$$
$$\mathrm{T_{75}}=1\oplus0\oplus{ \left[\begin{smallmatrix} 6&&7\\ &5 \end{smallmatrix}\right]}\oplus7\oplus6\oplus0, \mathrm{T_{76}}=0\oplus0\oplus{ \left[\begin{smallmatrix} 6&&7\\ &5 \end{smallmatrix}\right]}\oplus7\oplus6\oplus0$$
$$\mathrm{T_{77}}=1\oplus3\oplus{ \left[\begin{smallmatrix} 6&&7\\ &5 \end{smallmatrix}\right]}\oplus7\oplus6\oplus10, \mathrm{T_{78}}=0\oplus3\oplus{ \left[\begin{smallmatrix} 6&&7\\ &5 \end{smallmatrix}\right]}\oplus7\oplus6\oplus10$$
$$\mathrm{T_{79}}=1\oplus0\oplus{ \left[\begin{smallmatrix} 6&&7\\ &5 \end{smallmatrix}\right]}\oplus7\oplus6\oplus10, \mathrm{T_{80}}=0\oplus0\oplus{ \left[\begin{smallmatrix} 6&&7\\ &5 \end{smallmatrix}\right]}\oplus7\oplus6\oplus10$$
$$\mathrm{T_{81}}=1\oplus0\oplus5\oplus{ \left[\begin{smallmatrix} 6\\ &5 \end{smallmatrix}\right]}\oplus{ \left[\begin{smallmatrix} &7\\ 5 \end{smallmatrix}\right]}\oplus0,
\mathrm{T_{82}}=0\oplus0\oplus5\oplus{ \left[\begin{smallmatrix} 6\\ &5\end{smallmatrix}\right]}\oplus{ \left[\begin{smallmatrix} &7\\ 5 \end{smallmatrix}\right]}\oplus0$$
$$\mathrm{T_{83}}=1\oplus3\oplus5\oplus{ \left[\begin{smallmatrix} 6\\ &5 \end{smallmatrix}\right]}\oplus{ \left[\begin{smallmatrix} &7\\ 5 \end{smallmatrix}\right]}\oplus0,
\mathrm{T_{84}}=0\oplus3\oplus5\oplus{ \left[\begin{smallmatrix} 6\\ &5\end{smallmatrix}\right]}\oplus{ \left[\begin{smallmatrix} &7\\ 5 \end{smallmatrix}\right]}\oplus0$$
$$\mathrm{T_{85}}=1\oplus0\oplus5\oplus{ \left[\begin{smallmatrix} 6\\ &5 \end{smallmatrix}\right]}\oplus{ \left[\begin{smallmatrix} &7\\ 5 \end{smallmatrix}\right]}\oplus10,
 \mathrm{T_{86}}=0\oplus0\oplus5\oplus{ \left[\begin{smallmatrix} 6\\ &5\end{smallmatrix}\right]}\oplus{ \left[\begin{smallmatrix} &7\\ 5 \end{smallmatrix}\right]}\oplus10$$
$$\mathrm{T_{87}}=1\oplus3\oplus5\oplus{ \left[\begin{smallmatrix} 6\\ &5 \end{smallmatrix}\right]}\oplus{ \left[\begin{smallmatrix} &7\\ 5 \end{smallmatrix}\right]}\oplus10,
\mathrm{T_{88}}=0\oplus3\oplus5\oplus{ \left[\begin{smallmatrix} 6\\ &5\end{smallmatrix}\right]}\oplus{ \left[\begin{smallmatrix} &7\\ 5 \end{smallmatrix}\right]}\oplus10$$
$$\mathrm{T_{89}}=1\oplus3\oplus{ \left[\begin{smallmatrix} 6&&7\\ &5 \end{smallmatrix}\right]}\oplus7 \oplus{ \left[\begin{smallmatrix} &7\\ 5 \end{smallmatrix}\right]}\oplus0,
\mathrm{T_{90}}=0\oplus3\oplus{ \left[\begin{smallmatrix} 6&&7\\ &5\end{smallmatrix}\right]}\oplus7 \oplus{ \left[\begin{smallmatrix} &7\\ 5 \end{smallmatrix}\right]}\oplus0$$
$$\mathrm{T_{91}}=1\oplus3\oplus{ \left[\begin{smallmatrix} 6&&7\\ &5 \end{smallmatrix}\right]}\oplus{ \left[\begin{smallmatrix} 6\\ &5 \end{smallmatrix}\right]}\oplus6\oplus0,
\mathrm{T_{92}}= 0\oplus3\oplus{ \left[\begin{smallmatrix} 6&&7\\ &5 \end{smallmatrix}\right]}\oplus{ \left[\begin{smallmatrix} 6\\ &5 \end{smallmatrix}\right]}\oplus6\oplus0$$
$$\mathrm{T_{93}}=1\oplus0\oplus{ \left[\begin{smallmatrix} 6&&7\\ &5 \end{smallmatrix}\right]}\oplus7 \oplus{ \left[\begin{smallmatrix} &7\\ 5 \end{smallmatrix}\right]}\oplus0,
 \mathrm{T_{94}}=0\oplus0\oplus{ \left[\begin{smallmatrix} 6&&7\\ &5\end{smallmatrix}\right]}\oplus7 \oplus{ \left[\begin{smallmatrix} &7\\ 5 \end{smallmatrix}\right]}\oplus0$$
$$\mathrm{T_{95}}=1\oplus0\oplus{ \left[\begin{smallmatrix} 6&&7\\ &5 \end{smallmatrix}\right]}\oplus{ \left[\begin{smallmatrix} 6\\ &5 \end{smallmatrix}\right]}\oplus6\oplus0,
\mathrm{T_{96}}= 0\oplus0\oplus{ \left[\begin{smallmatrix} 6&&7\\ &5 \end{smallmatrix}\right]}\oplus{ \left[\begin{smallmatrix} 6\\ &5 \end{smallmatrix}\right]}\oplus6\oplus0$$
$$\mathrm{T_{97}}=1\oplus3\oplus{ \left[\begin{smallmatrix} 6&&7\\ &5 \end{smallmatrix}\right]}\oplus7 \oplus{ \left[\begin{smallmatrix} &7\\ 5 \end{smallmatrix}\right]}\oplus10,
\mathrm{T_{98}}=0\oplus3\oplus{ \left[\begin{smallmatrix} 6&&7\\ &5\end{smallmatrix}\right]}\oplus7 \oplus{ \left[\begin{smallmatrix} &7\\ 5 \end{smallmatrix}\right]}\oplus10$$
$$\mathrm{T_{99}}=1\oplus3\oplus{ \left[\begin{smallmatrix} 6&&7\\ &5 \end{smallmatrix}\right]}\oplus{ \left[\begin{smallmatrix} 6\\ &5 \end{smallmatrix}\right]}\oplus6\oplus10,
\mathrm{T_{100}}= 0\oplus3\oplus{ \left[\begin{smallmatrix} 6&&7\\ &5 \end{smallmatrix}\right]}\oplus{ \left[\begin{smallmatrix} 6\\ &5 \end{smallmatrix}\right]}\oplus6\oplus10$$
$$\mathrm{T_{101}}=1\oplus0\oplus{ \left[\begin{smallmatrix} 6&&7\\ &5 \end{smallmatrix}\right]}\oplus7 \oplus{ \left[\begin{smallmatrix} &7\\ 5 \end{smallmatrix}\right]}\oplus10,
\mathrm{T_{102}}=0\oplus0\oplus{ \left[\begin{smallmatrix} 6&&7\\ &5\end{smallmatrix}\right]}\oplus7 \oplus{ \left[\begin{smallmatrix} &7\\ 5 \end{smallmatrix}\right]}\oplus10$$
$$\mathrm{T_{103}}=1\oplus0\oplus{ \left[\begin{smallmatrix} 6&&7\\ &5 \end{smallmatrix}\right]}\oplus{ \left[\begin{smallmatrix} 6\\ &5 \end{smallmatrix}\right]}\oplus6\oplus10,
\mathrm{T_{104}}= 0\oplus0\oplus{ \left[\begin{smallmatrix} 6&&7\\ &5 \end{smallmatrix}\right]}\oplus{ \left[\begin{smallmatrix} 6\\ &5 \end{smallmatrix}\right]}\oplus6\oplus10$$
$$\mathrm{T_{105}}=1\oplus0\oplus{ \left[\begin{smallmatrix} 6&&7\\ &5 \end{smallmatrix}\right]}\oplus{ \left[\begin{smallmatrix} 6\\ &5 \end{smallmatrix}\right]}\oplus{ \left[\begin{smallmatrix} &7\\ 5 \end{smallmatrix}\right]}\oplus0,
\mathrm{T_{106}}=0\oplus0\oplus{ \left[\begin{smallmatrix} 6&&7\\ &5 \end{smallmatrix}\right]}\oplus{ \left[\begin{smallmatrix} 6\\ &5 \end{smallmatrix}\right]}\oplus{ \left[\begin{smallmatrix} &7\\ 5 \end{smallmatrix}\right]}\oplus0$$
$$\mathrm{T_{107}}=1\oplus3\oplus{ \left[\begin{smallmatrix} 6&&7\\ &5 \end{smallmatrix}\right]}\oplus{ \left[\begin{smallmatrix} 6\\ &5 \end{smallmatrix}\right]}\oplus{ \left[\begin{smallmatrix} &7\\ 5 \end{smallmatrix}\right]}\oplus0,
\mathrm{T_{108}}=0\oplus3\oplus{ \left[\begin{smallmatrix} 6&&7\\ &5 \end{smallmatrix}\right]}\oplus{ \left[\begin{smallmatrix} 6\\ &5 \end{smallmatrix}\right]}\oplus{ \left[\begin{smallmatrix} &7\\ 5 \end{smallmatrix}\right]}\oplus0$$
$$\mathrm{T_{109}}=1\oplus3\oplus{ \left[\begin{smallmatrix} 6&&7\\ &5 \end{smallmatrix}\right]}\oplus{ \left[\begin{smallmatrix} 6\\ &5 \end{smallmatrix}\right]}\oplus{ \left[\begin{smallmatrix} &7\\ 5 \end{smallmatrix}\right]}\oplus10,
\mathrm{T_{110}}=0\oplus3\oplus{ \left[\begin{smallmatrix} 6&&7\\ &5 \end{smallmatrix}\right]}\oplus{ \left[\begin{smallmatrix} 6\\ &5 \end{smallmatrix}\right]}\oplus{ \left[\begin{smallmatrix} &7\\ 5 \end{smallmatrix}\right]}\oplus10$$
$$\mathrm{T_{111}}=1\oplus0\oplus{ \left[\begin{smallmatrix} 6&&7\\ &5 \end{smallmatrix}\right]}\oplus{ \left[\begin{smallmatrix} 6\\ &5 \end{smallmatrix}\right]}\oplus{ \left[\begin{smallmatrix} &7\\ 5 \end{smallmatrix}\right]}\oplus10,
 \mathrm{T_{112}}=0\oplus0\oplus{ \left[\begin{smallmatrix} 6&&7\\ &5 \end{smallmatrix}\right]}\oplus{ \left[\begin{smallmatrix} 6\\ &5 \end{smallmatrix}\right]}\oplus{ \left[\begin{smallmatrix} &7\\ 5 \end{smallmatrix}\right]}\oplus10$$
$$ 2^{m-3}\times14=2^{6-3}\times14=112$$
By Theorem \ref{2.7}, then the number of tilting $\Gamma$-modules is 112 ($=2^{6-3}\times 14$).
\end{example}

\end{document}